\documentclass[a4paper,11pt,reqno]{amsart}

\usepackage{graphicx}
\usepackage{color}
\definecolor{rblue}{RGB}{39,64,139}
\usepackage{hyperref}
\hypersetup{pdftex,naturalnames=false,%
	draft=false,
	bookmarks=true,bookmarksopenlevel=1,%
	allcolors=rblue, 
	allbordercolors={.15 .25 .55},
	colorlinks=true,
	}
\usepackage{amssymb}
\usepackage[initials]{amsrefs}
\usepackage[cmtip]{xypic}

\title{Duflo Isomorphism and Chern-Weil Theory}
\author{Seunghun Hong}
\address{Northwestern College, 101 7th St SW, Orange City, Iowa, U.S.A.}
\email{seunghun.hong@gmail.com}
\urladdr{diracoperat.org}


\theoremstyle{plain}
\newtheorem{thm}{Theorem}[section]
\newtheorem{lem}[thm]{Lemma}
\newtheorem{prop}[thm]{Proposition}

\theoremstyle{definition}
\newtheorem{defn}[thm]{Definition}

\theoremstyle{remark}
\newtheorem*{rmk}{Remark}

\newcommand{\C}{\mathbb{C}}
\newcommand{\R}{\mathbb{R}}
\newcommand{\Z}{\mathbb{Z}}

\DeclareMathOperator{\tr}{tr}
\DeclareMathOperator{\Str}{Str}
\newcommand{\vol}{\mathrm{vol}}
\DeclareMathOperator{\Aut}{Aut}
\DeclareMathOperator{\Ind}{Ind}

\DeclareMathOperator{\ad}{ad}

\DeclareMathOperator{\Duf}{Duf}
\DeclareMathOperator{\CW}{CW}
\newcommand{\cupp}{\mathbin{\smallsmile}}
\DeclareMathOperator{\Spin}{Spin}
\DeclareMathOperator{\SO}{SO}

\newcommand{\spin}{\mathfrak{spin}}
\newcommand{\so}{\mathfrak{so}}
\DeclareMathOperator{\Cl}{Cl}

\newcommand{\topf}{\mathrm{top}}
\newcommand{\Hom}{\mathrm{Hom}}

\DeclareMathOperator{\End}{End}
\newcommand{\Id}{\mathbf{1}}
\newcommand{\bas}{\mathrm{bas}}
\DeclareMathOperator{\ev}{ev}

\begin{document}
\begin{abstract}
We explain how the distributional index of a transversally elliptic operator on a principal $G$-manifold $P$ that is obtained by lifting a Dirac operator on $P/G$ can serve as a link between the Duflo isomorphism and Chern-Weil forms. 
\end{abstract}
\maketitle

\section{Introduction}

In the introduction of their book, Berline, Getzler, and Vergne \cite{bgv} promotes the idea that index theory via heat kernels is a quantization of Chern-Weil theory. Meanwhile, Alekseev and Meinrenken's work \cite{alekmein} on quantum Weil algebra shows that the Duflo isomorphism is a quantization of distributions. We explain how these two points of view come together as we consider the distributional index of transversally elliptic operators developed by Atiyah and Singer \cite{atiyah}.

The typical context of our story is where we have a closed even-dimensional spin (or spin-c) manifold $M$. There one naturally has a principal $G$-bundle $P\to M$, where $G$ is some compact Lie group, and a Dirac operator $D$ acting on sections of an equivariant vector bundle $P\times_GE$. By picking a connection on $P$, one can horizontally lift the Dirac operator on $M$ to $P$. The resulting differential operator $\tilde D$, which may operate on $C^\infty(P)\otimes E$, is non-elliptic. However, it is transversally elliptic in the sense of Atiyah and Singer \cite{atiyah}*{Def.~1.3, p.~7}, which means in our situation that $\tilde D$ is elliptic in horizontal directions relative to the selected connection on $P$. Continuing with the idea of Atiyah and Singer,  one has the distributional index $[\tilde D]$ of $\tilde D$, which is a formal $\Z$-linear combination of characters of $G$. It is a theorem of Atiyah and Singer that $[\tilde D]$ is a genuine distribution on $G$ \cite{atiyah}*{Thm.~2.2, p.~10}. 

More is true in our case. Owing to the free $G$-action on $P$, the distribution $[\tilde D]$ is supported at the identity, and it can be identified as an element of the center $\mathcal{Z}(\mathfrak{g})$ of the universal enveloping algebra generated by the Lie algebra $\mathfrak{g}$ of $G$. Thus $[\tilde D]$ is subject to (the inverse of) the Duflo isomorphism 
\[ \Duf\colon S(\mathfrak{g})^G\to \mathcal{Z}(\mathfrak{g}). \]
Here $S(\mathfrak{g})^G$ denotes the $G$-invariant subalgebra of the symmetric algebra generated by $\mathfrak{g}$; it can be identified as the algebra of $G$-invariant distributions on $\mathfrak{g}$ that are supported at the origin. So it makes sense to pair $\Duf^{-1}[\tilde D]$ with a $G$-invariant analytic function $\varphi$ defined on some neighborhood of the origin in $\mathfrak{g}$. One can now ask whether the following equation holds:
\begin{equation}
 \langle \Duf^{-1}[\tilde D],\varphi \rangle = \langle \hat D\cupp \hat \varphi,\hat M \rangle.
 \label{eq:dufchweil}
 \end{equation}
Here $\hat D$ is the index class of $D$; $\hat \varphi$ is the characteristic class on $M$ obtained from $\varphi$ by the Chern-Weil homomorphism, and $\hat M$ is the fundamental homology class of $M$ determined by the spin structure. 

Our aim is to show that Equation \eqref{eq:dufchweil} holds under an equivariance condition on the connection associated with $D$. Equation \eqref{eq:dufchweil} explicitly shows how the distributional index of $\tilde D$ serves has a link between the Duflo isomorphism and the Chern-Weil forms.

\section{Setup}

\subsection{Basic Notations} Throughout this article, $M$ denotes a closed oriented Riemannian manifold of dimension $n$. Let $D$ be a Dirac operator acting on the sections of an equivariant vector bundle 
\[ P\times_GE\to M, \]
where $P$ is the total space of a principal bundle
\[ \kappa\colon P\to M\]
whose fibers are isomorphic to a compact Lie group $G$, with $G$ acting on fibers by right multiplication, and $E$ is a finite-dimensional complex $G$-vector space associated with a representation
\[ \nu\colon G\to  \Aut(E).\]
We denote by 
\[ \nu_*\colon \mathfrak{g}\to \End(E) \]
the Lie algebra representation induced by the differential of $\nu$ at the identity.

We assume that $E$ is a \emph{graded} $\Cl(n)$-module. Regarding the irreducible ones, there is only one (up to isomorphism) if $n$ is odd, and there are two if $n$ is even. If we fix one over the other for the even case, then we may speak of \emph{the} irreducible graded $\Cl(n)$-module $S$. Under such circumstances, $E$ is of the form 
\begin{equation}
 E=S\otimes W,
\label{eq:gclmd}
\end{equation}
where $\Cl(n)$ acts canonically on $S$ and trivially on $W$.

\subsection{Lift of a Dirac Operator} The existence of a Dirac operator on the sections of $P\times_GE$ implies that the vector bundle ${P\times_GE}$ admits a $\Cl(M)$-module structure. This means that $E$ is a module over the Clifford algebra $\Cl(n)$ generated by the Euclidean space $\R^n$ and that we have a bundle map
\begin{equation}
\begin{split} \xymatrix@R=4ex@C=0.5em{ \Cl(M) \ar[rr]^-c\ar[rd] && P\times_G\End(E) \ar[dl]\\
	&M
	}\end{split} 
\label{eq:cliffmds}
\end{equation}
such that, when restricted to each fiber, we have an isomorphism of algebras
\[ \Cl(T_xM) \xrightarrow{c} \Cl(n)\subset\End(E).\]
 
We assume that the Dirac operator $D$ is associated with a connection form $\theta$ on $P$ in the sense that $D$ is locally of the form
\[ D=\sum_{i=1}^nc(\xi_i)\nabla_{\xi_i}\]
for any local orthonormal frame $\{\xi_i\}_{i=1}^n$ for the tangent bundle $TM$ of $M$, where $\nabla$ is the covariant derivative induced by $\theta$ on the space $\Gamma(P\times_GE)$ of sections of $P\times_GE$. 

We define the \emph{lift} $\tilde D$ of $D$ as the differential operator on $\Gamma(P\times E)=C^\infty(P)\otimes E$ that is locally of the form 
\[ \tilde D=\sum_{i=1}^n c(\xi_i)\tilde\nabla_{\tilde \xi_i} \]
where $\tilde \xi_i$ denotes the horizontal lift of $\xi_i$, and $\tilde\nabla$ is the covariant derivative for the trivial bundle $P\times E\to P$, also induced by $\theta$ so that the action of $\tilde\nabla$ on the invariant subspace of $\Gamma(P\times E)$ agrees with $\nabla$ under the usual identification  $\Gamma(P\times E)^G\cong  \Gamma(P\times_GE)$.

\subsection{$\mathfrak{g}$-Spin Condition} The connection $\theta$ induces a $G$-invariant metric on $P$ in the following way: Let $\langle \ {,} \ \rangle_M$ be the pullback of the metric of $M$ along the bundle projection $\kappa$, and let  $\langle \ {,} \ \rangle_\mathfrak{g}$ be an inner product of our choice on $\mathfrak{g}$ that is invariant under the adjoint action of $G$ on $\mathfrak{g}$. Then we define the metric $\langle \ {,} \ \rangle$ on $P$ by
\[
 \langle v, w \rangle = \langle v, w\rangle_M + \langle \theta(v),\theta(w)\rangle_\mathfrak{g}.
\]

With the metric on $P$ at hand, we have the Riemannian connection $\nabla^P$ for $\mathfrak{X}(P):=\Gamma(TP)$. For each $X\in\mathfrak{g}$, denote the fundamental vector field it generates on $P$ by $\tilde X$. Let $\tilde\xi$ denote the vector field on $P$ that is the lift of a vector field $\xi$ on $M$.  Then $\nabla^P_{\tilde X}\tilde \xi$ is again the lift, call it $\tilde \xi'$, of a vector filed $\xi'$ on $M$. In fact, $\nabla^P_{\tilde X}\tilde \xi =\nabla^P_{\tilde \xi}\tilde X$ (see \cite{bgv}*{Lem.~5.2, p.~169}) and, as a consequence, the value of $\nabla^P_{\tilde X}\tilde \xi$ at $p\in P$ is completely determined by the values of $\tilde \xi$ and $\tilde X$ at $p$, which are in turn determined by $\xi$ at $x:=\kappa(p)$ and $X$. So there is a map 
\[ \omega_x\colon \mathfrak{g}\to \so(T_xM) \]
such that the value of $\nabla^P_{\tilde X}\tilde \xi$ at any point along the fiber of $x$ is the lift of the image of $\omega_x(X)$ operated on $\xi$ at $x$. As an example, if $P$ is a Lie group that has $G$ as a subgroup, then $\omega$ is $1/2$ times the adjoint action of $\mathfrak{g}$ on the orthogonal complement of $\mathfrak{g}$ in the Lie algebra of $P$. For this reason, we are more interested in the linear endomorphism $2\omega_x$. There is another significance of the operator $2\omega_x$. Consider the exponential map at $p\in P$, which we denote by
\[ \exp_{p}\colon T_pP\to P.\]
For $X\in\mathfrak{g}$, let $\tilde X_p$ denote the value of $\tilde X$ at $p$. The differential of $\exp_p$ at $\tilde X_p$ satisfies
\begin{equation}
 \exp_{p*, \tilde X_p}(\tilde \xi_p) = \biggl( \ell_{p\cdot \exp(X)}\circ \frac{1-e^{2\omega_x(X)}}{2\omega_x(X)} \biggr)(\xi_x) 
\label{eq:jacpbdl}
\end{equation}
for all $\xi_x\in T_xM$, where $\ell_{p\cdot\exp(X)}$ denotes the operator that lifts vectors in $T_xM$ horizontally into $T_{p\cdot \exp(X)}P$. Equation \eqref{eq:jacpbdl} follows from a generic expression for differentials of exponential maps on Riemannian manifolds; see \mbox{Duistermaat} \cite{duistermaathk}*{Lem.~9.5, p.~110}.

At this point we introduce an equivariance condition for $\nabla^P$, namely, that
\begin{equation}
  [\nu_*(X),c(\xi_x)] = c(2\omega_x(X)\xi_x)
\label{eq:kequivcliff}
\end{equation}
holds for all $X\in\mathfrak{g}$ and all $\xi_x\in T_xM$. We shall refer to this condition as (for lack of better words) the  \emph{$\mathfrak{g}$-spin condition} for $\theta$. It implies, in particular, that $2\omega_x$ is a Lie algebra homomorphism for each $x\in M$. Using the bundle map $c$ to identify $T_xM$ with $\R^n$, the Lie algebra homomorphism $2\omega_x$ induces a Lie algebra representation
\[ \alpha_x\colon \mathfrak{g}\to \so(n).
 \label{eq:cliffso}
\]
This representation is equivalent to the Lie algebra representation
\begin{equation}
 \gamma_x\colon \mathfrak{g} \to \spin(n)\subset\Cl(n)
\label{eq:liegamma}
\end{equation}
defined by the relation
\begin{equation}
 \alpha_x(X)(v) = [\gamma_x(X),v]
\label{eq:cliffrep}
\end{equation}
for all $X\in\mathfrak{g}$ and all $v\in \R^n$. Here the bracket on the right-hand side denotes commutation in $\Cl(n)$. 

In terms of the Lie algebra representation $\gamma_x$,  the $\mathfrak{g}$-spin condition can be written as
\begin{equation}
[\nu_*(X),c(\xi_x)] = [ \gamma_x(X),c(\xi_x)].
\label{eq:gcliffeqvar}
\end{equation}
This commutation relation implies that the $\mathfrak{g}$-action on $E$ respects the factorization \eqref{eq:gclmd}, so that
\[ \nu_*=\gamma_x\otimes\Id+\Id\otimes\tau_x \]
for some Lie algebra representation 
\[ \tau_x\colon \mathfrak{g}\to \End(W). \]

\begin{rmk}
The $\mathfrak{g}$-spin condition is always satisfied if the Lie algebra representation $\gamma_x$ is induced by a Lie group representation $G\to\Spin(n)$, in which case, we have $TM\cong P\times_G\R^n$, where $G$ acts on $\R^n$ through the double covering $\Spin(n)\to\SO(n)$. That includes the case where $G=\Spin(n)$ and $P$ is a spin structure for $M$.
\end{rmk}

\section{Distributional Index}

The lifted operator $\tilde D$ on $P$ is by construction a transversally elliptic operator on $E$-valued functions on $P$. Because the connection $\theta$ is $G$-invariant, $\tilde D$ is also $G$-invariant,  so the kernel of $\tilde D$ is a $G$-space. Denote the even and odd subspaces of the kernel as $\ker(\tilde D_+)$ and $\ker(\tilde D_-)$, respectively. Let $\hat G$ be the unitary dual of $G$, that is, the set of unitary equivalence classes of irreducible unitary representations of $G$. For each $\rho\in \hat G$, let $V_\rho$ denote the representation space of any representation in the class of $\rho$. It is a result of Atiyah and Singer \cite{atiyah}*{Lem.~2.3, p.~10} that the pairings 
\[ \langle \ker \tilde D_\pm, V_\rho \rangle := \dim \Hom_G(\ker \tilde D_\pm, V_\rho) \]
are finitely valued for all $\rho\in\hat G$. So 
\[ [\ker \tilde D_\pm] := \sum_{\rho\in\hat G} \langle \ker \tilde D_\pm, V_\rho \rangle \rho \]
are well-defined elements of the formal representation group 
\[ \hat R(G):=\prod_{\rho\in\hat G}\Z\cdot \rho. \]
Moreover, identifying $\rho$ with its character $\chi_\rho$, the formal sums $[\ker \tilde D_\pm]$ converge in the distributional sense \cite{atiyah}*{Thm.~2.2, p.~10}. The \emph{distributional index} of $\tilde D$ is then defined as
\begin{align}
 [\tilde D]:= \ &[\ker \tilde D_+]-[\ker \tilde D_-] \notag \\
= \ & \sum_{\rho\in\hat G}(\langle \ker \tilde D_{+},V\rangle-\langle \ker \tilde D_{-},V\rangle)\chi_\rho. \label{eq:distind}
\end{align}

\begin{prop}\label{prop:dtindprg}
Let $f$ be a smooth class function on $G$. Let $P_t$ be the heat kernel associated with the (scalar) Laplacian on $P$. Then
\begin{equation}
\langle [\tilde D] , f \rangle = \int_M\int_G\Str(P_{t}(p,p\cdot g)f(g)\nu(g^{-1}))\,dg\,dx. 
\label{eq:dtindprg}
\end{equation} 
Here $\Str$ denotes the super trace for graded operators, and $p$ is an arbitrary point in the fiber over $x$.
\end{prop}
\begin{rmk}
Equation \eqref{eq:dtindprg} shows that $[\tilde D]$ indeed has point support at the identity, owing to the finite-propagation property of the heat kernel (see \cite{roe}*{Prop.~7.24, p.~107}).
\end{rmk}
\begin{proof}
We identify the domain $\Gamma(P\times E)$ of $\tilde D$ with the following space: 
\[ \mathcal{D}(P,E) := C^\infty(P)\otimes E.\]
We can decompose this space compatibly with regards to the sum  \eqref{eq:distind}:
\[ \mathcal{D}(P,E) \cong \bigoplus_{\rho\in \hat G} \mathcal{D}(P,E)_\rho, \]
where $\mathcal{D}(P,E)_\rho$ is the $\rho$-isotypic component of $\mathcal{D}(P,E)$. In other words,
\[ \mathcal{D}(P,E)_\rho = (C^\infty(P)\otimes E\otimes V_{\check\rho})^G\otimes V_\rho.\]
Here $\check\rho$ denotes the contragredient of $\rho$.  Note that the operator $\tilde D$ acts trivially on the factors in $V_\rho$ or $V_{\check\rho}$ that appears  above. Denote by $\tilde D_{\rho}$ the restriction of $\tilde D$ to $(C^\infty(P)\otimes E\otimes V_\rho)^G$. Let $D_{\rho}$ be  the Dirac operator $D$ twisted by $V_{\rho}$. Then the action of $\tilde D_{\rho}$ matches with that of $D_\rho$ under the usual identification $(C^\infty(P)\otimes E\otimes V_\rho)^G\cong \Gamma(P\times_G(E\otimes V))$. Thus, the coefficient of $\chi_\rho$ in Equation \eqref{eq:distind} equals the ordinary index of $D_\rho$. In other words,
\[  \langle \ker \tilde D_{+},V\rangle-\langle \ker \tilde D_{-},V\rangle = \Ind(D_\rho).\]

Because $D_\rho$ is elliptic and Fredholm, so is $\tilde D_\rho$. Moreover, if we write the  Laplacian on $P$ as $\Delta_P$, then $\tilde D_\rho^2+\Delta_P$ is equal to a constant operator $C_\rho$, namely, the Casimir associated with $\rho$ (see \cite{bgv}*{Prop.~5.6, p.~172}). So $L_\rho:=-\Delta_P+C_\rho$ is a generalized Laplacian that agrees with $\tilde D_\rho^2$ on $(C^\infty(P)\otimes E\otimes V_\rho)^G$. Let $P_{t}$ be the heat kernel associated with $\Delta_P$, that is, the integral kernel of the operator $e^{-t\Delta_P}$ for $t\in(0,\infty)$. Owing to the McKean-Singer formula (see \cite{bgv}*{Thm.~3.50, p.~124; Prop.~5.7, p.~173}), 
\[ \Ind(D_\rho) = \Str(e^{-tL_\rho})=e^{-tC_\rho}\int_M\int_G\Str(P_t(p,p\cdot g)\nu(g)^{-1})\chi_{\check\rho}(g)\,dg\,dx. \]
Here $p$ is a point of our choice in the fiber over $x$; the integral is well-defined owing to the equivariance of the heat kernel $P_t$, which in turn follows from the equivariance of the Laplacian $\Delta_P$; see \cite{bgv}*{\S~5.2}. Moreover, as we shall see later in Proposition \ref{prop:locindrho}, the integral is of $O(1)$  for $t\to0+$. So we may drop the factor $e^{-tC_\rho}$ from the above equation.

In short, we may express $[\tilde D]$ as follows:
\begin{equation*}
 [\tilde D]  = \sum_{\rho\in\hat G} \biggl(\int_M\int_G\Str(P_{t}(p,p\cdot g)\nu(g^{-1}))\chi_\rho(g)\,dg\,dx\biggr) \chi_\rho.
\label{eq:distindker}
\end{equation*} 
Then the pairing of $[\tilde D]$ with a function $f\in C^\infty(G)$ can be calculated as follows:
\[ 
\langle [\tilde D] , f \rangle = \sum_{\rho\in\hat G} \biggl(\int_M\int_G\Str(P_{t}(p,p\cdot g)\nu(g^{-1}))\chi_\rho(g)\,dg\,dx\biggr)\langle \chi_\rho,f\rangle. 
\]
If $f$ is a class function, then, since the characters of $G$ form an orthonormal basis for square-integrable class functions on $G$, we have $\sum_{\rho\in G} \chi_\rho(g)\langle \chi_\rho,f\rangle = f(g)$. Therefore, we have the desired equation \eqref{eq:dtindprg}.
\end{proof}

\section{Chern-Weil Forms}
We quickly recall the construction of the Chern-Weil homomorphism.  As a preliminary remark, suppose we have a formal power series $\varphi\in \R[[\mathfrak{g}^*]]$. Let $\wedge(N)$ be the exterior algebra generated by some finite-dimensional vector space $N$ over $\R$, and  let $\wedge^+ (N)$ be its subalgebra generated by elements of even degree. Then the formal power series $\varphi$ defines a map $\mathfrak{g}\otimes\wedge^+(N)\to\wedge^+(N)$ in the following way. Identify $\mathfrak{g}$ with $\mathfrak{g}\otimes\{1\}\subset \mathfrak{g}\otimes\wedge^+(N)$. By duality, the evaluation of $\chi\in\mathfrak{g}^*$ at an arbitrary element $\eta=\sum X_j\otimes\eta_j \in\mathfrak{g}\otimes\wedge^+(N)$ takes the value $\chi(\eta)=\sum \chi(X_j)\eta_j$. The evaluation map $\ev_\eta\colon \mathfrak{g}^*\to \wedge^+(N)$, $\chi\mapsto \chi(\eta)$, extends uniquely as an algebra homomorphism to $\ev_\eta\colon \R[[\mathfrak{g}^*]]\to \wedge^+(N)$. Then $\ev_\eta(\varphi)=:\varphi(\eta)$ is the evaluation of $\varphi$ at $\eta$. Note that $\ev_\eta$ factors through $S(\mathfrak{g})$. All of this makes sense even if we replace $N$ with $N_\C:= N\otimes\C$.

Now let $\Omega(P)$ denote as usual the algebra of differential forms on $P$. The pullback 
\[
 \kappa^*\colon \Omega(M)\to \Omega(P)
\]
induced by the bundle projection is an injective algebra homomorphism; its image is the algebra $\Omega_{\bas}(P)$ of basic forms on $P$. So $\kappa^*$ has a left inverse (push-forward), which we denote by
\[
 \kappa_*\colon \Omega_\bas(P)\to \Omega(M).
\]
The curvature $\Theta$ of our connection $\theta$ is an element of $\mathfrak{g}\otimes\Omega^+_{\bas}(P)$, so it makes sense to evaluate a formal power series $\varphi\in\R[[\mathfrak{g}^*]]$ at $\Theta/2\pi i\in \mathfrak{g}\otimes\Omega^+_\bas(P)_\C$. The resultant $\varphi(\Theta/2\pi i)$ is a basic form on $P$, and we can apply the push-forward $\kappa_*$ to it. This process yields an algebra homomorphism, namely,
\[ \begin{array}{cccc}
 \CW\colon & \R[[\mathfrak{g}^*]]^\mathfrak{g}&\to&\Omega^+(M)_\C,\\
 &\varphi&\mapsto& \kappa_*\varphi(\Theta/2\pi i).
 \end{array}
\]
We refer to $\CW(\varphi)$ as the \emph{Chern-Weil form} of $\varphi$. The high point of Chern-Weil theory is that the de~Rham cohomology class of the Chern-Weil form $\CW(\varphi)$ is a characteristic class.

As demonstrated by Berline and Vergne \cite{berlinevergne}, a similar ``construction'' occurs in heat kernel calculations. This is because there is a vector space isomorphism
\[ \spin(n) \cong \wedge^2(\R^n)\]
by virtue of the \emph{Chevalley map} 
\[ \sigma\colon \Cl(n)\to \wedge(\R^n), \]
which is defined, in terms of the standard orthonormal basis $\{e_i\}_{i=1}^{n}$ of $\R^n$, by the equation
\[ \sigma(e_{i_1}\dotsb e_{i_k}) = e_{i_1}\wedge\dotsb\wedge e_{i_k} \]
for any subset $\{e_{i_1},\dotsc,e_{i_k}\}$ of the basis. The Chevalley map is a vector space isomorphism, and the image of $\spin(n)$ is exactly $\wedge^2(\R^n)$. 

The calculation that mimics the construction of the Chern-Weil map is captured in Lemma \ref{lem:intcw} below. But first, we set up some notations.

\begin{defn}
For each $x\in M$, we denote by 
\[ \lambda_x\colon \mathfrak{g}\to \wedge^2(\R^n) \]
the composition of the Lie algebra homomorphism \eqref{eq:liegamma} with the Chevalley map. This may be identified as an element of $\mathfrak{g}^*\otimes\wedge(\R^n)$. Applying the isomorphism $\mathfrak{g}^*\cong\mathfrak{g}$ by the inner product, we get the element
\[ \Lambda := \sum_{i=1}^{\dim\mathfrak{g}} X_i\otimes \lambda(X_i)\in \mathfrak{g}\otimes\wedge^+(\R^n),\]
where $\{X_i\}_{i=1}^{\dim\mathfrak{g}}$ is any orthonormal basis for $\mathfrak{g}$ (the definition does not depend on the choice of the basis). 
\end{defn}

\begin{lem}\label{lem:intcw}
Let $h_t$ be the Euclidean heat kernel (the Gaussian function) on $\mathfrak{g}$. Let $\varphi$ be an analytic function defined near the origin of $\mathfrak{g}$. Let $\psi$ be a $G$-invariant bump function supported within the domain of $\varphi$. Then the $\wedge(\R^n)$-valued function
\[ t  \mapsto   \int_{\mathfrak{g}} h_t(X) \psi(X) \varphi(X)e^{-\lambda(X)}\,dX\]
has an asymptotic expansion $\sum^\infty_{k=0}\Psi_kt^k$ for $t\to0+$. (The asymptotic expansion is independent of the choice of $\psi$.) The $k$th coefficient $\Psi_k$ is contained in $\bigoplus_{q=0}^k\wedge^{2q}(\R^n)$. If $k\le n/2$, then the component of $\Psi_k$ of degree $2k$ (the highest degree part) is equal to that of $\varphi(-2\Lambda)\in \wedge^+(\R^n)$.
\end{lem}
\begin{proof}[Remark on the proof]
This lemma is similar in form to Lemma~11.3 in Duistermaat \cite{duistermaathk}*{p.~137}. The proof given there can be carried over almost verbatim. The only extra thing that needs to be checked is that $\sum_{i=1}^{\dim\mathfrak{g}}\lambda(X_i)\lambda(X_i)=0$ when $\{X_i\}_{i=1}^{\dim\mathfrak{g}}$ is an orthonormal basis for $\mathfrak{g}$; this can be verified using the Jacobi identity of the Lie bracket. We omit the details.
\end{proof}

The element $\varphi(-2\Lambda)$ appearing in Lemma \ref{lem:intcw} is a Chern-Weil form in disguise (provided that $\varphi$ is $G$-invariant), as implied by the next lemma. Before stating the lemma, recall that we have a smooth map $c\colon\Cl(M)\to \Cl(n)$ that is an algebra isomorphism when restricted to the fiber over any $x\in M$. This yields, via the Chevalley identification, a smooth map $\wedge (TM)\to \wedge(\R^n)$ that is a vector space isomorphism when restricted to the fiber over $x\in M$. Though it is an abuse of notation, we shall denote this map also as
\[ 
c\colon \wedge(TM)\to \wedge(\R^n).
\]

One more notation; consider the map $\sharp\colon \Omega^1(M)\to \mathfrak{X}(M)$ that maps a $1$-form to its dual vector field relative to the metric. This induces an algebra isomorphism 
\[
 \sharp\colon \Omega(M) \to \wedge\mathfrak{X}(M),
\]
which maps differential forms to polyvector fields (so-called the ``raising of indices'').

\begin{lem}\label{lem:chernweil}
Consider the polyvector field $(\kappa_*\Theta)^\sharp$, namely, the one obtained by taking the push-forward $\kappa_*\Theta$ of the curvature form $\Theta$ along the bundle projection and then raising its indices. Denote the value of this polyvector field at an arbitrary point $x\in M$ by  $(\kappa_*\Theta)^\sharp_x$. Then 
$c((\kappa_*\Theta)^\sharp_x)=2\Lambda$.
\end{lem}
\begin{rmk}
A straightforward consequence is that, for any $\varphi\in \R[[\mathfrak{g}^*]]^\mathfrak{g}$ and any $x\in M$, we have
\begin{equation}
 \mathcal{A}(\varphi) = \CW(\varphi)_x,
\label{eq:cwev}
\end{equation}
where $\mathcal{A}$ is the following composition of algebra homomorphisms:
\[ \R[[\mathfrak{g}^*]]^\mathfrak{g} \to \wedge(\R^n)\to \wedge T_xM \to \wedge T_x^*M, \]
where the first map is the evaluation at $\Lambda/\pi i$; the second is the inverse of $\wedge(T_xM)\xrightarrow{c}\wedge(\R^n)$; and the last is the lowering of indices.
\end{rmk}
\begin{proof}
For $X\in\mathfrak{g}$, let $\langle X,\Lambda\rangle$ denote the inner product of $X$ with the $\mathfrak{g}$-factors of $\Lambda$, so that $\langle X,\Lambda\rangle=\lambda(X)$. It is sufficient to check that $\langle X,2\Lambda\rangle= c\langle X, (\kappa_*\Theta)_x^\sharp\rangle$, that is, $2\lambda(X)=c\langle X, (\kappa_*\Theta)_x^\sharp\rangle$. 

Relation \eqref{eq:cliffrep} dictates that 
\[ \gamma_x(X) = -\frac{1}{2}\sum_{i<j}\langle e_i,\alpha_x(X)e_j\rangle e_ie_j, \]
where $\{e_i\}_{i=1}^n$ is the standard orthonormal basis for $\R^n$. This implies that
\[ \lambda(X) = -\frac{1}{2}\sum_{i<j} \langle e_i,\alpha_x(X)e_j\rangle e_ie_j.\]

Now let  $\xi_i$ denote the image of $e_i$ under the inverse of $\Cl(T_xM)\xrightarrow{c} \Cl(n)$. Let $\tilde \xi_i$ denote the horizontal lift of $\xi_i$ at, say, $p$ in the fiber of the projection $\kappa\colon P\to M$ over $x$. Then
\[ c\langle X,\kappa_*\Theta_x^\sharp\rangle =\sum_{i<j} \langle X, \Theta_p(\tilde\xi_i,\tilde\xi_j)\rangle e_ie_j.\]
The curvature form $\Theta$ and the Riemannian connection $\nabla^P$ for $\mathfrak{X}(P)$ satisfy (see \cite{duistermaathk}*{Lem.~9.1, p.~102})
\[ \langle X, \Theta_p(\tilde\xi_i,\tilde\xi_j)\rangle = 2\langle \nabla_{\tilde X}\tilde \xi_i,\tilde \xi_j \rangle.\]
The right-hand side equals, by definition, $\langle \alpha_x(X)e_i,e_j\rangle$. This proves that $c\langle X,\kappa_*\Theta_x^\sharp\rangle = 2\lambda(X)$ as desired.
\end{proof}

Combining Lemmas \ref{lem:intcw} and \ref{lem:chernweil}, we have what is essentially the local index theorem for Dirac operators:
\begin{prop}\label{prop:locindrho}
Assume the notation in Proposition \ref{prop:dtindprg}. Let  $j_\mathfrak{g}$ and $j_M$ be the analytic functions on $\mathfrak{g}$ defined by:
\[
 j_\mathfrak{g}(X) = \det\nolimits^{1/2}\biggl[\frac{\sinh(\ad(X/2))}{{\ad}(X/2)}\biggr],\qquad
 j_M(X) = \det\nolimits^{1/2}\biggl[\frac{\sinh(\alpha(X)/2)}{\alpha(X)/2}\biggr].
 \]
Let $f$ be a smooth class function on $G$. Fix a point $x\in M$, and let $p$ an arbitrary point in the fiber over $x$. Let 
\[ I(t) := \int_G\Str(P_t(p,p\cdot g)\nu(g)^{-1})f(g)\,dg.\]
This integral does not depend on the choice of $p$. Moreover, for $t\to0+$, the integral is of $O(t^{1/2})$ if $n$ (the dimension of $M$) is odd; if $n$ is even, then
\[
 I(t) =\frac{1}{(4\pi)^{n/2}} \Str(F_x(-2\Lambda))+O(t),
\]
where $F_x$ is the following product of analytic functions on $\mathfrak{g}$:
\[
 F_x := j_\mathfrak{g} j_M^{-1} \tr(e^{-\tau_x}) \exp^*(f).
\] 
\end{prop}
\begin{proof}
We have already observed in Proposition \ref{prop:dtindprg} that the integral $I(t)$ is independent of the choice of $p$ in the fiber over $x$. To prove the rest of our claim, we change the domain of the integral $I(t)$ from $G$ to $\mathfrak{g}$ by means of the exponential map. As it is well-known, $j_\mathfrak{g}^2(X)$ calculates the Jacobian determinant of the exponential map when the exponential map is diffeomorphic near $X$. So
\begin{align*}
 I(t)&=\int_\mathfrak{g}\Str\bigl( P_t(p,p\cdot \exp(X))
	e^{-\nu_*(X)}\bigr)j_\mathfrak{g}^2(X)
	f(\exp(X))
	\,dX\\
&= \int_\mathfrak{g} \Str\bigl(P_t(p,p\cdot\exp X)  e^{-\gamma(X)}\bigr)
j_\mathfrak{g}^2(X)\tr(e^{-\tau_x(X)})
	f(\exp(X))\,dX.
\end{align*}
Note that we have used the relation $\nu_*=\gamma+\tau$. 

The asymptotic expansion of $P_t$ for $t\to0+$ is well-known. See, for instance, Berline, Getzler, and Vergne \cite{bgv}*{Thm.~5.8, p.~174}, from which we deduce that 
\[
I(t) \sim \frac{1}{(4\pi t)^{n/2}} \sum_{m=0}^\infty t^m \Str\Bigl( \int_\mathfrak{g} h_t(X) \Phi_m(x,X) e^{-\lambda(X)}\,dX\Bigr),
\]
where $h_t$ is the Gaussian function on $\mathfrak{g}$, and 
\[ \Phi_0(x,X)=j_\mathfrak{g}(X)j_M^{-1}(X)\tr(e^{-\tau_x(X)})
	f(\exp(X)). \]
The rest of the argument proceeds similarly to that found in Berline and Vergne's proof \cite{berlinevergne} of the Atiyah-Singer index theorem (see \cite{bgv}*{$\S$~5.4}). In short, by Lemma \ref{lem:intcw} and the representation theory of Clifford algebras, we conclude that
\[ I(t)=O(t^{1/2}) \]
if $n$ is odd, and
\begin{align*}
 I(t) &= \frac{1}{(4\pi t)^{n/2}}\Str\Bigl(\int_\mathfrak{g}h_t(X)\Phi_0(x,X)e^{-\lambda(X)}\,dX\Bigr)+O(t)\\
 &=\frac{1}{(4\pi)^{n/2}} \Str(\Phi_0(x,-2\Lambda))+O(t) 
\end{align*}
if $n$ is even. 
\end{proof}

\section{Proof of the Theorem}
\begin{thm}\label{thm:dufcwind}
Let $M$ be a closed oriented Riemannian manifold of dimension $n$. Let $P$ be a principal bundle over $M$ whose fibers are isomorphic to a compact Lie group $G$. Let $E$ be a $G$-vector space that is also a graded $\Cl(n)$-module. Let $D$ be a Dirac operator on $\Gamma(P\times_GE)$ associated with a connection $\theta$ on $P$. Let $\tilde D$ be the lift of $D$. Suppose the connection $\theta$ satisfies the $\mathfrak{g}$-spin condition \eqref{eq:kequivcliff}. Then the distributional index $[\tilde D]$ of $\tilde D$ satisfies
\begin{equation}
 \langle \Duf^{-1}[\tilde D],\varphi \rangle = \langle \hat D\cupp\hat \varphi,\hat M\rangle
 \label{eq:distindduf}
\end{equation}
for any $G$-invariant analytic function $\varphi$ defined on some neighborhood of the origin in $\mathfrak{g}$, where $\hat D$ is the index class of $D$; $\hat \varphi$ is the characteristic class on $M$ obtained from $\varphi$ by the Chern-Weil homomorphism, and $\hat M$ is the fundamental homology class of $M$ determined by the orientation. 
\end{thm}

\begin{proof}
We begin by recalling the definition of the Duflo isomorphism in terms of distributions. Let $\mathcal{E}'(\mathfrak{g})^G_0$ denote the algebra of $G$-invariant distributions on $\mathfrak{g}$ supported at the origin. Likewise, let  $\mathcal{E}'(G)^G_e$ denote the algebra of $G$-invariant distributions on $G$ supported at the identity. The Duflo isomorphism is defined as
\[
 \Duf={\exp_*}\circ j\colon \mathcal{E}'(\mathfrak{g})^G_0\to \mathcal{E}'(G)^G_e,
\]
where $\exp_*$ denotes the push-forward induced by the exponential map, and $j$ denotes the multiplication by $j_\mathfrak{g}$. 

Since the exponential map is a local diffeomorphism on some neighborhood $U$ of the origin in $\mathfrak{g}$, there is an isomorphism  
\[ \log^*\colon C^\infty(U) \to C^\infty(\exp[U])\]
to which the pullback along the exponential map serves as a left inverse. This induces, by duality, a linear map
\[ \log_*\colon \mathcal{E}'(G)^G_e\to \mathcal{E}'(\mathfrak{g})^G_0,\]
which is inverse to the push-forward $\exp_*$ along the exponential map. Then we have $\Duf^{-1} = j^{-1} \circ \log_*$, so
\[
\langle \Duf^{-1} [\tilde D],\varphi  \rangle = \langle [\tilde D],\log^*(j_\mathfrak{g}^{-1}\psi\varphi) \rangle.
\]
Here we have included a suitable $G$-invariant bump function $\psi$  so that the pullback $\log^*(j_\mathfrak{g}^{-1}\psi\varphi)$ makes sense. This is fine since $[\tilde D]$ has point support. Invoking Equation \eqref{eq:dtindprg}, we have
\[
\langle\Duf^{-1}[\tilde D],\varphi\rangle
	= \int_{M}\int_G \Str\bigl( P_t(p,p\cdot g)
	\log^*(j_\mathfrak{g}^{-1}\varphi\psi)(g)\nu(g)^{-1} \bigr)
	\,dg\,dx. 
\]
Because the left-hand side is independent of $t$, it is sufficient to show that the right-hand side is asymptotically equal to $\langle \hat D\cupp \hat\varphi,\hat M\rangle$ as $t\to0+$ when $n$ is even, and that  it is of $O(t^{1/2})$ when $n$ is odd. 

To that end, we focus on the integral over $G$, namely, 
\[ J(t):= \int_G\Str\bigl( P_t(p,p\cdot g)
	\log^*(j_\mathfrak{g}^{-1}\varphi\psi)(g)\nu(g)^{-1} \bigr)
	\,dg.
	\]
Apply Proposition \ref{prop:locindrho} with $f$ substituted with $\log^*(j_{\mathfrak{g}}^{-1}\varphi\psi)$.  The conclusion is that 
\[ J(t)=O(t^{1/2}) \]
if $n$ is odd, and that
\[
J(t) = \frac{1}{(4\pi)^{n/2}} \Str(F_x(-2\Lambda))+O(t)
\]
if $n$ is even, where $F_x$ is the following product of analytic functions on $\mathfrak{g}$:
\[  F_x= j_\mathfrak{g} j_M^{-1}\tr(e^{-\tau_x}) \varphi. \]
Note that the bump function $\psi$ does not appear in the asymptotic expansion; this is because, in the limit $t\to0+$, the function $t\mapsto P_t(x,x\cdot g)$ is of $O(t^\infty)$ if $g$ is outside any neighborhood of the identity; so the bump function may be dropped without affecting the asymptotic behavior. 

Let $\vol$ denote the Riemannian volume form on $M$ associated with the measure $dx$ on $M$. Owing to Equation \eqref{eq:cwev}, we have 
\begin{equation*}
J(t)\,\vol = O(t^{1/2})
\end{equation*}
if $n$ is odd; if $n$ is even, then
\begin{equation}
J(t) \,\vol = \bigl.\CW(j_{M}^{-1}\tr(e^{-\tau}))\CW(\varphi)\bigr|^{\topf} +O(t), \label{eq:distindasymp2}
\end{equation} 
where the decoration $\left.\right|^{\topf}$ picks out the top degree part of the differential form at hand. Note that we have dropped the factor $j_\mathfrak{g}$ in Equation \eqref{eq:distindasymp2} because $j_\mathfrak{g}(-2\Lambda)=1$ (this can be checked by a routine calculation similar to the one involved in the proof of Lemma \ref{lem:intcw}). Since the de~Rham cohomology class of the Chern-Weil form $\CW(j_{M}\tr(e^{-\tau}))$ is the index class of $D$, we indeed have Equation \eqref{eq:distindduf}.
\end{proof}

\section*{Acknowledgements}

The author thanks Nigel Higson for suggesting Equation \eqref{eq:dufchweil} and also for the helpful conversations. This research was partially supported under NSF grant DMS-1101382.


\begin{bibdiv} 
\begin{biblist}


\bib{alekmein}{article}{
   author={Alekseev, A.},
   author={Meinrenken, E.},
   title={The non-commutative Weil algebra},
   journal={Invent. Math.},
   volume={139},
   date={2000},
   number={1},
   pages={135--172},
   issn={0020-9910},
   review={\MR{1728878 (2001j:17022)}},
   doi={10.1007/s002229900025},
}


\bib{atiyah}{book}{
   author={Atiyah, Michael Francis},
   title={Elliptic operators and compact groups},
   series={Lecture Notes in Mathematics, Vol. 401},
   publisher={Springer-Verlag},
   place={Berlin},
   date={1974},
   pages={ii+93},
   review={\MR{0482866 (58 \#2910)}},
}


\bib{bgv}{book}{
   author={Berline, Nicole},
   author={Getzler, Ezra},
   author={Vergne, Mich{\`e}le},
   title={Heat kernels and Dirac operators},
   series={Grundlehren Text Editions},
   note={Corrected reprint of the 1992 original},
   publisher={Springer-Verlag},
   place={Berlin},
   date={2004},
   pages={x+363},
   isbn={3-540-20062-2},
   review={\MR{2273508 (2007m:58033)}},
}

\bib{berlinevergne} {article}{
   author={Berline, Nicole},
   author={Vergne, Mich{\`e}le},
   title={A computation of the equivariant index of the Dirac operator},
   language={English, with French summary},
   journal={Bull. Soc. Math. France},
   volume={113},
   date={1985},
   number={3},
   pages={305--345},
   issn={0037-9484},
   review={\MR{834043 (87f:58146)}},
}

\bib{duistermaathk}{book}{
   author={Duistermaat, J. J.},
   title={The heat kernel Lefschetz fixed point formula for the spin-$c$
   Dirac operator},
   series={Modern Birkh\"auser Classics},
   note={Reprint of the 1996 edition},
   publisher={Birkh\"auser/Springer, New York},
   date={2011},
   pages={xii+247},
   isbn={978-0-8176-8246-0},
   review={\MR{2809491 (2012b:58032)}},
   doi={10.1007/978-0-8176-8247-7},
}

\bib{roe}{book}{
   author={Roe, John},
   title={Elliptic operators, topology and asymptotic methods},
   series={Pitman Research Notes in Mathematics Series},
   volume={395},
   edition={2},
   publisher={Longman},
   place={Harlow},
   date={1998},
   pages={ii+209},
   isbn={0-582-32502-1},
   review={\MR{1670907 (99m:58182)}},
}


\end{biblist}
\end{bibdiv}
\end{document}